\documentclass[11pt,a4paper]{amsart}
\usepackage{subcaption}
\usepackage{graphicx}
\usepackage[colorlinks]{hyperref}
\usepackage{pinlabel}
\usepackage{color}
\usepackage{cleveref}

\usepackage{todonotes}

\title{The Arnoux--Yoccoz mapping classes via Penner's construction}
\author{Livio Liechti}
\thanks{The first author was supported by the Swiss National Science Foundation~(grant nr.\ 175260)}
\address{D\'epartement de Math\'ematiques, Universit\'e de Fribourg, Chemin du Mus\'ee 23, 1700 Fribourg, Suisse}
\email{livio.liechti@unifr.ch}
\author{Bal\'azs Strenner}
\address{Georgia Institute of Technology, School of Mathematics, Atlanta GA
  30332, USA}
\email{strenner@math.gatech.edu}
\email{strennerb@gmail.com}

\newcommand{\R}{\mathbf{R}}
\newcommand{\Q}{\mathbf{Q}}
\newcommand{\C}{\mathbf{C}}
\newcommand{\calF}{\mathcal{F}}

\newtheorem{theorem}{Theorem}[section]

\newtheorem{proposition}[theorem]{Proposition}
\newtheorem{lemma}[theorem]{Lemma}

\begin{document}

\begin{abstract}
  We give a new description of the Arnoux--Yoccoz mapping classes as a product of
  two Dehn twists and a finite order element. The construction is analogous to Penner's
  construction of mapping classes with small stretch factors.
\end{abstract}

\maketitle

\section{Introduction}

The mapping class group of a surface~$S$ is the group of isotopy classes of
orientation-preserving homeomorphisms of~$S$. Motivated by studying geometric
structures on 3-manifolds, Thurston~\cite{Thurston88} modernized the theory of
mapping class groups in the 1970s by giving a classification of elements into
three types: finite order, reducible and pseudo-Anosov. This article is
concerned about the third type. A mapping class is \emph{pseudo-Anosov} if it
has a representative homeomorphism~$\phi$ and singular measured foliations~$\calF^u$ 
and~$\calF^s$ on~$S$ such that~$\phi(\calF^u) = \lambda \calF^u$ and
$\phi(\calF^s) = \lambda^{-1}\calF^s$ for some~$\lambda >
1$. The number~$\lambda$ is independent of choice of the
representative homeomorphism and it is called the \emph{stretch factor} or
\emph{dilatation} of the pseudo-Anosov mapping class.

Thurston showed that stretch factors of pseudo-Anosov mapping classes of the
closed orientable surface~$S_g$ are algebraic integers with degree bounded
above by~$6g-6$. He claimed without proof in~\cite{Thurston88} that the degree~$6g-6$ 
was realizable, but this statement was only recently proven in~\cite{StrennerDegrees}. 
For some time, however, even the fact that
pseudo-Anosov stretch factors of arbitrarily large degrees exist was not
justified. This fact was first shown by Arnoux and Yoccoz~\cite{ArnouxYoccoz81}
in 1981. They constructed a pseudo-Anosov mapping class~$\tilde{h}_g$ on $S_g$
for each~$g\ge 3$ with a stretch factor of algebraic degree~$g$. We will recall
the construction in \Cref{sec:ay-or} and give more reasons for why
the Arnoux--Yoccoz examples are of importance after stating the main results.

Despite the mapping classes~$\tilde{h}_g$ being probably the single most widely
studied explicit family of pseudo-Anosov mapping classes, to this day, no
constructions were known other than the original approach by Arnoux
and Yoccoz.

The goal of the paper is to present a new description as a product of two Dehn
twists and a finite order mapping class. 
We hope that this new description will shed new light on the examples and help
construct new analogous families of mapping classes that might also serve as
interesting examples. An alternative description of the
Arnoux--Yoccoz mapping classes was also asked for by Margalit in Section 10 
of~\cite{MargalitProblems}.

\begin{theorem}
\label{thm:arnoux-yoccoz}
The Arnoux--Yoccoz mapping class~$\tilde{h}_g$ on the surface~$S_g$ is
conjugate to~$\tilde{f}_g = r \circ T_a \circ T_b^{-1}$, where~$T_a$ and~$T_b^{-1}$ 
are positive and negative Dehn twists about the curves~$a$ and~$b$
pictured on \Cref{fig:ay_penner}, and~$r$ is a rotation of the surface by one
click in either direction.
  \begin{figure}[htb]
    \labellist
    \small\hair 2pt
    \pinlabel {$a$} at 33 8
    \pinlabel {$b$} at 58 10
    \endlabellist
    \centering
    \includegraphics[width=0.4\textwidth]{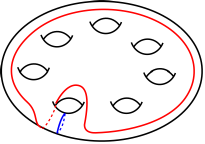}
    \caption{The surface~$S_g$ with a rotational symmetry of order~$g$. This
      figure shows the case~$g=7$.}
    \label{fig:ay_penner}
  \end{figure}
\end{theorem}

For the proof, we use the fact, shown by the second author in~\cite[Section
5]{StrennerSAF}, that the mapping classes~$\tilde{h}_g$ arise as lifts of
mapping classes on nonorientable surfaces. More precisely, there is a
pseudo-Anosov mapping class~$h_g$ (see \Cref{sec:ay-nonor} for the definition)
on the closed nonorientable surface~$N_{g+1}$ of genus~$g+1$ for each~$g\ge 3$
such that~$\tilde{h}_g$ is the lift of~$h_g$ by the double cover~$S_g \to N_{g+1}$. 
We will deduce \Cref{thm:arnoux-yoccoz} from the following.

\begin{theorem}
\label{thm:arnoux-yoccoz-nonor}
The nonorientable Arnoux--Yoccoz mapping class~$h_g$ on the surface~$N_{g+1}$
is conjugate to~$f_g = r \circ T_c$, where~$T_c$ is a Dehn twist about the two-sided
curve~$c$ pictured on \Cref{fig:ay_penner-nonor}, and~$r$ is a rotation of the
surface by one click in either direction.
  \begin{figure}[htb]
    \labellist
    \small\hair 2pt
    \pinlabel {$c$} at 57 60
    \endlabellist
    \centering
    \includegraphics[scale=0.7]{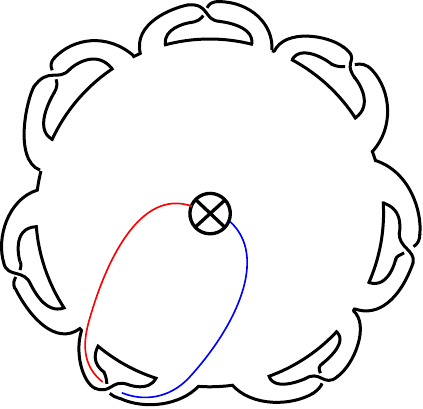}
    \caption{The circle with an X inside it indicates a crosscap: the inside of
      the circle is not part of the surface and antipodal points of the circle
      are identified. A disk with one crosscap is therefore a M\"obius
      strip. So this figure shows a nonorientable surface obtained by attaching~$g$ 
      twisted bands to the boundary of a M\"obius strip. The surface has one
      boundary component. By gluing a disk to the boundary component, we obtain
      the closed surface~$N_{g+1}$.}
    \label{fig:ay_penner-nonor}
  \end{figure}
\end{theorem}

The direction of twisting about~$T_c$ is important, see \Cref{fig:nonor-twist}
later for the reason behind this. On a nonorientable surface, there is no
notion of positive or negative twisting, so we specify the direction of the
Dehn twist~$T_c$ by the coloring of the curve~$c$ on \Cref{fig:ay_penner-nonor}
as follows. By cutting out the crosscap in the middle and cutting the twisted
bands, we obtain an orientable surface with an embedding in~$\R^2$ coming from
the figure. Our cut-up surface inherits the orientation of~$\R^2$. The blue and
red parts of our curve indicate the parts where the twisting behaves like a
positive and negative twist, respectively, with respect to the orientation of
the cut-up surface.

In \Cref{prop:power-from-penner} we show that both~$f_g^g$ and~$\tilde{f}_g^g$
arise from Penner's construction. In this sense, the mapping classes~$f_g$ and~$\tilde{f}_g$ 
are analogous to the pseudo-Anosov mapping classes with small
stretch factors constructed by Penner in~\cite{Penner91}.

\subsection*{History and motivation}

The Sah--Arnoux--Fathi invariant (in short, SAF invariant) is an invariant of
interval exchange transformations and measured foliations that measures certain
dynamical properties. In genus~2, the invariant foliations of all pseudo-Anosov
mapping classes have non-vanishing SAF invariant~\cite{McMullen03, Calta04}.
The Arnoux--Yoccoz examples were the first examples for pseudo-Anosov mapping
classes in genus~3 and higher whose invariant foliations had vanishing SAF
invariants. For more on pseudo-Anosov maps with vanishing SAF invariant, 
see~\cite{ArnouxSchmidt09, CaltaSchmidt13, DoSchmidt16, StrennerSAF}.

The Arnoux--Yoccoz examples are interesting also because they are known not to
arise from Thurston's construction~\cite{Thurston88}, the first general
construction of pseudo-Anosov mapping classes. This was shown by Hubert and
Lanneau~\cite{HubertLanneau06}, using the fact that the extension 
fields~$\Q(\lambda_g+\lambda_g^{-1})$ are not totally real, where~$\lambda_g$ 
denotes the stretch factor of~$\tilde{h}_g$. Since the examples that arise from
Thurston's construction have certain special properties (for one, see the next
paragraph), the Arnoux--Yoccoz mapping classes have become valuable exotic
examples.

A further motivation for studying the Arnoux--Yoccoz examples comes from
Teichm\"uller theory. The two invariant measured foliations of a pseudo-Anosov
mapping class give rise to a \emph{singular Euclidean metric} or \emph{flat
  structure} on the surface. The \emph{Veech group} of a flat surface is the
group of its affine symmetries. The infinite cyclic group formed by the powers
of the pseudo-Anosov map is always part of the Veech group. It is not known,
however, whether the Veech group can possibly equal this infinite cyclic 
group~\cite[Problem 6]{HubertMasurSchmidtZorich06}. In many cases (for example, for
pseudo-Anosov maps arising from Thurston's construction), the Veech group is
known to contain free groups, hence it is known not to be cyclic. Since the
Arnoux--Yoccoz examples do not arise from Thurston's construction, presumably
they might be good candidates for cyclic Veech groups. In the~$g=3$ case,
however, Hubert, Lanneau and M\"oller~\cite{HubertLanneauMoller09} found
additional elements in the the Veech group of the Arnoux--Yoccoz flat surface,
hence that Veech group is not cyclic. It remains an open question whether the
Arnoux--Yoccoz surfaces have cyclic Veech groups when~$g\ge 4$.

For other work on the Arnoux--Yoccoz mapping classes and their flat
surfaces, see~\cite{Arnoux88, Bowman10, Bowman13,McMullen15, HooperWeiss18}.

\subsection*{Generalizations}

The Arnoux--Yoccoz example in the $g=3$ case has been generalized by Arnoux and
Rauzy in Section 3 of~\cite{ArnouxRauzy91}, see also Section~4.2 
of~\cite{PoggiaspallaLowensteinVivaldi08}. On the surface~$N_4$, one example in
the Arnoux--Rauzy family is the third power of the Arnoux--Yoccoz mapping 
class~$h_3$, which is conjugate to~$T_{r^{-2}(c)} \circ T_{r^{-1}(c)}\circ T_c$ by
our \Cref{thm:arnoux-yoccoz-nonor}. We believe that other members of
Arnoux--Rauzy family include the mapping classes~$T_{r^{-2}(c)} \circ T_{r^{-1}(c)}\circ T_c^k$ 
where~$k\ge 1$, but we will not give a proof of this.

In a follow-up paper~\cite{LiechtiStrennerMinimal}, we will generalize the
twist-and-rotation construction in \Cref{thm:arnoux-yoccoz-nonor} in a
different way in order to construct minimal pseudo-Anosov stretch factors on
various different nonorientable surfaces. In particular, we will show that the
Arnoux--Yoccoz example~$h_3$ has minimal stretch factor on~$N_4$ among
pseudo-Anosov mapping classes with an orientable invariant foliation.

\subsection*{Acknowledgements}

We are grateful to Pierre Arnoux, Dan Margalit and Thomas Schmidt for helpful
comments on an earlier version of this paper. We also thank the referees for
their constructive feedback.

\section{Background}
\label{sec:background}

\subsection{The orientable Arnoux--Yoccoz examples}
\label{sec:ay-or}

In this section we recall the original construction of the orientable
Arnoux--Yoccoz mapping classes~$\tilde{h}_g$ from~\cite{ArnouxYoccoz81}. We
give this description in order to provide context only. The content of this
section will not be used in the proofs.

Fix some~$g\ge 3$ and let~$\alpha$ be the unique real number in~$(0,1)$ 
satisfying~$\alpha + \cdots + \alpha^g = 1$. We construct a measured 
foliation~$\calF$ on the surface~$S_g$ as follows. Start with the rectangle on
\Cref{fig:orientable-ay} foliated by vertical leaves. Equip~$\calF$ with a
measure so that the width of the rectangle is 2. Identify the two vertical
sides to obtain a foliated annulus. Divide the top and bottom boundary
components of this annulus to~$2g$ intervals each as shown on
\Cref{fig:orientable-ay} and identify each interval on the top side with the
interval on the bottom side that has the same length so that shaded rectangles
above the core of the annulus are joined to shaded rectangles below the core of
the annulus and empty rectangles are joined to empty rectangles. We obtain a
measured foliation~$\calF$ on the surface~$S_g$. The two-sided simple closed
curve~$\gamma$ obtained from the core of the annulus is transverse to~$\calF$
and the first return map of~$\calF$ with respect to~$\gamma$ induces the
subdivision of~$S_g$ into the~$2g$ rectangles shown on
\Cref{fig:orientable-ay}.

\begin{figure}
  \centering
  \begin{tikzpicture}[scale = 12]
    \def\rad{0.004} \def\colorstrength{40} \def\wrappingindex{8}
    \def\height{0.1}

\def\midpointarray{0.1296975/1/1/0.2594/$\alpha$, 0.3890925/1/2/0.2594/$\alpha$, 0.5860758/1/3/0.1346/$\alpha^2$, 0.7206474/1/4/0.1346/$\alpha^2$, 0.8228404/1/5/0.0698/$\alpha^3$, 0.8926548/1/6/0.0698/$\alpha^3$, 0.9456715/1/7/0.0362/$\alpha^4$, 0.9818905/1/8/0.0362/$\alpha^4$, 0.0860758/-1/4/0.1346/$\alpha^2$, 0.2206474/-1/3/0.1346/$\alpha^2$, 0.3228404/-1/6/0.0698/$\alpha^3$, 0.3926548/-1/5/0.0698/$\alpha^3$, 0.4456715/-1/8/0.0362/$\alpha^4$, 0.4818905/-1/7/0.0362/$\alpha^4$, 0.6296975/-1/2/0.2594/$\alpha$, 0.8890925/-1/1/0.2594/$\alpha$}

\def\endpointarray{{0/1/black, 0.259395/1/black/gray/0/, 0.5187901/1/black/white/0/, 0.6533616/1/black/gray/0/, 0.7879332/1/black/white/0/, 0.8577476/1/black/gray/0/, 0.927562/1/black/white/0/, 0.963781/1/black/gray/0/, 1/1/black/white/0/}, {0.0187901/-1/black/gray/0/3.62190772036, 0.1533616/-1/black/white/0/, 0.2879332/-1/black/gray/0/, 0.3577476/-1/black/white/0/, 0.427562/-1/black/gray/0/, 0.463781/-1/black/white/0/, 0.5/-1/black/gray/0/, 0.759395/-1/black/white/0/, 0.0187901/-1/black/gray/0/3.62190772036}}


\foreach \list in \endpointarray{
\def\lastx{}
\def\lastsign{}

\foreach \x/\sign/\notused/\col/\isflipped/\mp[count=\i from 0,remember=\x as \lastx, remember=\sign as \lastsign] in \list{
\ifnum \i > 0

\def\darkcolor{\col!\colorstrength!white}
\ifnum \isflipped = 0 
\def\bp{\colorstrength}
\def\ep{0}
\else
\def\bp{0}
\def\ep{\colorstrength}
\fi

\ifnum \wrappingindex = \i

\ifnum \sign = -1
\fill[left color=\darkcolor, right color=\darkcolor]  (\lastx, 0) rectangle (1,\lastsign*\height);
\fill[left color=\darkcolor, right color=\darkcolor]  (0,0) rectangle (\x, \sign*\height);
\else
\fill[left color=\darkcolor, right color = \darkcolor] (\lastx, 0) rectangle (\x, \sign*\height);
\fi

\else

\fill[left color=\darkcolor, right color = \darkcolor] (\lastx, 0) rectangle (\x, \sign*\height);

\fi
\fi
}
}


\foreach \x in {1,...,200} {
  \draw[ultra thin, gray] (\x/200, -\height) -- (\x/200, \height);
}

\newcommand\definepos[1]{
\ifnum #1 = 1
\def\pos{above}
\else
\def\pos{below}
\fi
}

\foreach \list in \endpointarray{
\foreach \x/\sign in \list{
\draw (\x,0) -- (\x,\sign*\height);
}
}

\begin{scope}[red, ultra thick]
  \draw (0.12, \height) .. controls +(0.05,-0.05) and +(-0.05,-0.05) .. (0.4, \height);
  \draw (0.63, -\height) .. controls +(0.05,+0.05) and +(-0.05,+0.05) .. (0.9, -\height);
\end{scope}

\foreach \sign in {-1,1}{
\draw[dashed] (0,\sign*\height) -- +(1,0);
}

\draw[ultra thick] (0,0) -- (1,0);

\foreach \list in \endpointarray{
\foreach \x/\sign/\singcol in \list{
\filldraw[fill=\singcol, draw=black] (\x,\sign*\height) circle (\rad);
}
}

\foreach \x/\sign/\label/\length/\text in \midpointarray{
\definepos{\sign}
\node at (\x,\sign*\height) [\pos] {\text};
}

\node[right] at (1,0) {$\gamma$};
\node at (0.17, 0.05) {$\gamma'$};

\end{tikzpicture}
  \caption{Construction of the mapping class~$\tilde{h}_4$ by Arnoux and Yoccoz.}
  \label{fig:orientable-ay}
\end{figure}

The key observation is that the two-sided simple closed curve~$\gamma'$ on
\Cref{fig:orientable-ay} is also transverse to~$\calF$, has length~$2\alpha$
and the first return map of~$\calF$ induces a decomposition into rectangles
which is isomorphic to the original decomposition, up to scaling the measure by
a factor of~$\alpha$. Therefore there are homeomorphisms of the surface that
map~$\gamma$ to~$\gamma'$ and~$\calF$ to~$\frac1{\alpha}\calF$. These
homeomorphisms are all isotopic via a leaf-preserving isotopy, therefore they
belong to the same mapping class. The mapping class~$\tilde{h}_g$ is defined to
be this mapping class.

\subsection{The nonorientable Arnoux--Yoccoz examples}
\label{sec:ay-nonor}

In~\cite{StrennerSAF}, the second author has constructed a mapping class~$h_g$
on the closed nonorientable surface~$N_{g+1}$ of genus~$g+1$ in a way that is
analogous to the construction described in \Cref{sec:ay-or}.

\begin{figure}[ht]
  \centering
  \begin{tikzpicture}[scale=12]
    \def\rad{0.004}
    \def\colorstrength{40}
    \def\wrappingindex{2}
    \def\height{0.1}
    \def\midpointarray{0.259395/1/1/0.2594, 0.7781851/1/1/0.2594, 0.1721517/-1/2/0.1346, 0.4412948/-1/2/0.1346, 0.6456808/-1/3/0.0698, 0.7853096/-1/3/0.0698, 0.891343/-1/4/0.0362, 0.963781/-1/4/0.0362}

    \def\endpointarray{{0/1/black, 0.5187901/1/black/yellow/0/, 0.0375801/-1/black/yellow/0/3.6218987807, 0.3067233/-1/black/green/0/, 0.5758664/-1/black/green/0/, 0.7154952/-1/black/blue/0/, 0.855124/-1/black/blue/0/, 0.927562/-1/black/cyan/0/, 1/-1/black/cyan/0/}}

    \foreach \list in \endpointarray{
      \foreach \x/\sign in \list{
        \draw (\x,0) -- (\x,\sign*\height);
      }
    }

    \foreach \sign in {-1,1}{
      \draw[dashed] (0,\sign*\height) -- +(1,0);
    }

    \draw[ultra thick] (0,0) -- (1,0);

    \foreach \list in \endpointarray{
      \foreach \x/\sign/\singcol in \list{
        \filldraw[fill=\singcol, draw=black] (\x,\sign*\height) circle (\rad);
      }
    }


\foreach \x in {1,...,200} {
  \draw[ultra thin, gray] (\x/200, -\height) -- (\x/200, \height);
}

\begin{scope}[red, ultra thick]
  \draw (0.2593, \height) .. controls +(0.05,-0.05) and +(-0.05,-0.05) .. (0.7781, \height);
\end{scope}

  \path (0,\height) -- node[above] {$\alpha$} (0.518,\height);
  \path (0.518,\height) -- node[above] {$\alpha$} (1.03,\height);
  \path (0.03,-\height) -- node[below] {$\alpha^2$} (0.3,-\height);
  \path (0.3,-\height) -- node[below] {$\alpha^2$} (0.57,-\height);
  \path (0.57,-\height) -- node[below] {$\alpha^3$} (0.71,-\height);
  \path (0.71,-\height) -- node[below] {$\alpha^3$} (0.85,-\height);
  \path (0.85,-\height) -- node[below] {$\alpha^4$} (0.92,-\height);
  \path (0.92,-\height) -- node[below] {$\alpha^4$} (1,-\height);

  \node[right] at (1,0) {$\gamma$};
  \node[above] at (0.35,0.02) {$\gamma'$};

  \end{tikzpicture}
  \caption{Construction of the mapping class~$h_4$ by the second author.}
  \label{fig:AY-nonor}
\end{figure}

Consider the rectangle on \Cref{fig:AY-nonor} together with the vertical
measured foliation of width 1, but now identify the two vertical sides with a
flip to obtain a foliated M\"obius strip. Divide the boundary component of
length~2 to intervals as shown on the figure and identify pairs of intervals of
equal length by translations.

Let~$\gamma$ be the core of the M\"obius strip and let~$\gamma'$ be the
one-sided curve in the figure that is also transverse to the foliation~$\calF$.
As before, the first return maps of the foliation on~$\gamma$ and~$\gamma'$
each induce decompositions of the surface into~$g$ foliated rectangles, and there
are homeomorphisms that map~$\gamma$ to~$\gamma'$ and~$\calF$ 
to~$\frac1{\alpha}\calF$. Once again, all these homeomorphisms are isotopic hence
they define the same mapping class. The mapping class~$h_g$ is defined to be
this mapping class.

Our approach for proving that a mapping class constructed in a different way is
conjugate to~$h_g$ is to use the fact (immediate from the construction) that
the triple~$(\calF, \gamma, \gamma')$ uniquely determines~$h_g$ up to
conjugation. To be more specific, we state this fact formally as follows.

\begin{lemma}\label{prop:foliation-and-curve-determine}
  Let~$g\ge 3$ and let~$f_1$ and~$f_2 = h_g$ be pseudo-Anosov mapping classes
  with unstable measured foliations~$\calF_1$ and~$\calF_2$ on the surface~$N_{g+1}$. 
  Let~$\gamma_1$ and~$\gamma_2$ be one-sided simple closed curves on~$N_{g+1}$, 
  transverse to~$\calF_1$ and~$\calF_2$, respectively. If there is a
  homeomorphism~$\phi: N_{g+1} \to N_{g+1}$ such 
  that~$\phi(\calF_1) = \calF_2$,~$\phi(\gamma_1) = \gamma_2$ 
  and~$\phi(f_1(\gamma_1)) = f_2(\gamma_2)$, then~$f_1$ and~$f_2$ are conjugate in
  the mapping class group of~$N_{g+1}$.
\end{lemma}

\subsection{Penner's construction of pseudo-Anosov mapping classes}
\label{sec:penner}

Consider the annulus~$A = \{z \in \C: 1\le |z| \le 2\}$ and define the positive
Dehn twist in~$A$ by the formula~$T(z) = z\cdot e^{2\pi i(1-|z|)}$. Given a
two-sided simple closed curve~$c$ in a surface~$S$, its \emph{marking} is an
embedding~$\phi_c: A \to S$ where~$c$ is the image of the circle~$\{z: |z|=\frac32\}$. 
The Dehn twist about the marked curve~$(c,\phi_c)$ is
defined as~$\phi_c \circ T \circ \phi_c^{-1}$ on~$\phi_c(A)$ and as the
identity otherwise. When the marking is clear from the context, we denote this
Dehn twist simply by~$T_c$. Note that if the surface~$S$ is orientable, then~$T_c$ 
is a positive twist if~$\phi_c$ is orientation-preserving and a negative
twist if~$\phi_c$ is orientation-reversing.

Two marked two-sided simple closed curves~$c_1$ and~$c_2$ are said to
\emph{intersect inconsistently} if~$\phi_{c_1}^{-1} \circ \phi_{c_2}$ is
orientation-reversing at all points where the composition is defined. A pair of
simple closed curves on a surface is in \emph{minimal position} if one cannot
decrease their intersection number by isotoping them. A collection of curves on
a surface is said to \emph{fill} the surface if they are in pairwise minimal
position and the complementary regions of the curves are disks and
once-punctured disks.

Penner gave the following construction of pseudo-Anosov mapping classes in
\cite{Penner88} (see also \cite{Fathi92}).

\begin{theorem}[Penner's construction]
  Let~$c_1, \ldots, c_n$ be a filling collection of pairwise inconsistently
  intersecting marked two-sided simple closed curves on a surface~$S$. Then any
  product of the Dehn twists~$T_{c_1}, \ldots, T_{c_n}$ containing each twist
  at least once is pseudo-Anosov.
\end{theorem}

For an orientable surface, the filling and the inconsistently intersecting
property implies that the collection of curves is a union of two multicurves~$\Gamma_1$ 
and~$\Gamma_2$ and the allowable products contain positive twists
about the curves in~$\Gamma_1$ and negative twists about the curves in~$\Gamma_2$.

One nice property of Penner's construction is that the stretch factor and the
invariant measured foliations are straightforward to compute. We explain the
process briefly here, for more details, 
see~\cite{Penner88,Penner91,Fathi92,StrennerDegrees}. By smoothing out the
intersections of the collection~$\{c_1, \ldots, c_n\}$, one obtains a bigon
track~$\tau$ that is invariant under each twist~$T_{c_i}$, hence under any
product of them as well. This process is illustrated on
\Cref{fig:smoothing-intersections}. Each curve~$c_i$ is carried on~$\tau$ and
its characteristic measure~$\mu_i$ is a 0-1-valued measure on~$\tau$ that takes
the value~1 on the branches of~$\tau$ traversed by~$c_i$ and the value~0 on the
other branches. The cone~$C$ generated by the measures~$\mu_i$ is invariant
under the Dehn twists~$T_{c_i}$. Moreover, a product~$f$ of~$T_{c_i}$
containing all twists at least once acts on~$C$ by a Perron--Frobenius matrix
(a matrix with nonnegative entries that has a power whose entries are all
positive). The largest eigenvalue of this matrix has multiplicity~1 and it is
the stretch factor of~$f$. The corresponding eigenvectors define a positive
measure on~$\tau$ that is unique up to scaling, and this train track measure
defines the unstable foliation of~$f$.

\begin{proposition}\label{prop:power-from-penner}
  The mapping classes~$f_g^g$ and~$\tilde{f}_g^g$ arise from Penner's construction.
\end{proposition}
\begin{proof}
  We have
  \begin{displaymath}
    f_g^g = T_{r^{-(g-1)}(c)} \circ \cdots \circ T_c.
  \end{displaymath}
  and
  \begin{displaymath}
    \tilde{f}_g^g = T_{r^{-(g-1)}(a)}\circ T_{r^{-(g-1)}(b)}^{-1}\circ \cdots \circ
    T_a \circ T_b^{-1} .
  \end{displaymath}

  \Cref{fig:all-curves} shows that the marked curve~$c$ and its rotated copies
  intersect inconsistently (red intersects blue at every intersection).
  Furthermore, any pair of curves intersects exactly once and hence minimally,
  and the complement of the union of the curves consists of discs. 
  Hence~$f_g^g$ indeed arises from Penner's construction.

  In the second case,~$A = \{a, \ldots, r^{g-1}(a)\}$ 
  and~$B = \{b, \ldots, r^{g-1}(b)\}$ are filling multicurves, and we twist only
  positively along curves in~$A$ and negatively along curves in~$B$, 
  hence~$\tilde{f}_g^g$ also arises from Penner's construction.
\end{proof}

\section{Proofs}
\label{sec:nonor}

In this section, we give the proofs of \Cref{thm:arnoux-yoccoz} and
\Cref{thm:arnoux-yoccoz-nonor}.

\begin{proof}[Proof of \Cref{thm:arnoux-yoccoz-nonor}]
  By \Cref{prop:foliation-and-curve-determine}, the proof reduces to the
  study of the unstable foliation of~$f_g$ and the image of the core
  curve~$\gamma$ of the crosscap under~$f_g$. 

  First we describe the unstable foliation of~$f_g$. Although~$f_g$ does not
  arise from Penner's construction, its~$g$th power does, since
  \begin{displaymath}
    f_g^g = T_{r^{-(g-1)}(c)} \circ \cdots \circ T_c.
  \end{displaymath}
  The invariant foliations of~$f_g$ and its powers are the same, therefore we
  may use the process described in \Cref{sec:penner} for~$f_g^g$ to find the
  unstable foliation of~$f_g$.

  \Cref{fig:all-curves} shows the curve~$c$ and its iterates under the rotation~$r$ 
  and \Cref{fig:bigon-track-circular} shows the invariant bigon track of~$f_g^g$ 
  obtained by smoothing out the intersections of these curves. Note
  that this bigon track is invariant not only under all twists~$T_{r^i(c)}$,
  but also under~$r$.

  \begin{figure}[htb]
    \begin{subfigure}[ht]{0.45\linewidth}
      \labellist
      \footnotesize\hair 2pt
      \pinlabel {$c$} at 55 63
      \pinlabel {$r(c)$} [ ] at 30 93
      \pinlabel {$r^{-1}(c)$} [ ] at 127 26
      \endlabellist
      \centering
      \includegraphics[width=\textwidth]{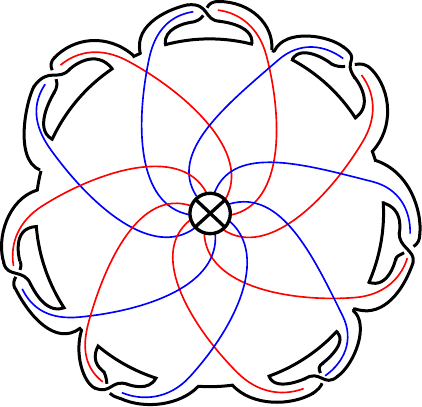}
      \caption{  }
      \label{fig:all-curves}
    \end{subfigure}\hfill
    \begin{subfigure}[ht]{0.45\linewidth}
      \labellist
      \small\hair 2pt
      \pinlabel {$1$} at 35 15
      \pinlabel {$1$} at 70 -5
      \pinlabel {$\lambda^6$} at 3 102
      \pinlabel {$\lambda^6$} at 7 42
      \pinlabel {$\lambda^{5}$} at 33 181
      \pinlabel {$\lambda^{5}$} at 4 131
      \pinlabel {$\lambda^{4}$} at 71 192
      \pinlabel {$\lambda^{4}$} at 128 198
      \pinlabel {$\lambda^{3}$} at 151 187
      \pinlabel {$\lambda^{3}$} at 195 142
      \pinlabel {$\lambda^{2}$} at 203 113
      \pinlabel {$\lambda^{2}$} at 200 55
      \pinlabel {$\lambda$} at 175 18
      \pinlabel {$\lambda$} at 135 -5
      \pinlabel {$\gamma$} at 88 104
      \pinlabel {$\gamma'$} at 30 72
      \endlabellist
      \centering
      \includegraphics[width=\textwidth]{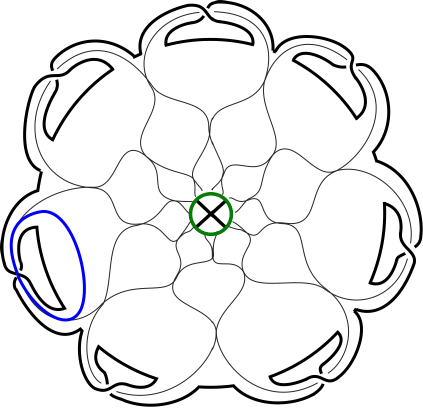}
      \caption{  }
      \label{fig:bigon-track-circular}
    \end{subfigure}
    \caption{}
    \label{fig:smoothing-intersections}
  \end{figure}

  For~$i=0,\dots,g-1$, let~$\mu_i$ be the characteristic measure of the curve~$r^{-i}(c)$. 
  The mapping class~$T_{c}$ acts on the cone~$C$ generated by the
  characteristic measures~$\mu_i$ by the matrix which has~$1$s on the diagonal
  and in the first row, and~$0$s otherwise, and~$r$ acts by a permutation
  matrix. So, for example, when~$g=7$, the acting matrix of~$f_g=r\circ T_c$ on
  the cone~$C$ is
  $$ M_g = \begin{pmatrix}
    0 & 1 & 0 & 0 & 0 & 0 & 0\\
    0 & 0 & 1 & 0 & 0 & 0 & 0\\
    0 & 0 & 0 & 1 & 0 & 0 & 0\\
    0 & 0 & 0 & 0 & 1 & 0 & 0\\
    0 & 0 & 0 & 0 & 0 & 1 & 0\\
    0 & 0 & 0 & 0 & 0 & 0 & 1\\
    1 & 1 & 1 & 1 & 1 & 1 & 1\\
  \end{pmatrix},$$ the companion matrix of the polynomial~$x^7-x^{6}-\cdots-x-1.$ 
  In general, the characteristic polynomial is~$x^g-x^{g-1}-\cdots-x-1.$ 
  The action of~$f_g^g$ on~$C$ is given by~$M_g^g$,
  so by the construction described in \Cref{sec:penner}, the stretch factor and
  the unstable foliation of~$f_g^g$ are the largest eigenvalue and the
  corresponding eigenvector for~$M_g^g$. As a consequence, the stretch factor
  and the unstable foliation of~$f_g$ are given by the largest eigenvalue and
  the corresponding eigenvector for~$M_g$. Hence the stretch factor~$\lambda$
  of~$f_g$ is the largest real root of~$x^g-x^{g-1}-\cdots-x-1$. The
  corresponding eigenvector is~$(1,\lambda,\lambda^2,\dots,\lambda^{g-1})$,
  therefore the unstable foliation of~$f_g$ is given by the
  measure~$\sum_{i=0}^{g-1} \lambda^{i} \mu_i$ on our bigon track. Note 
  that~$\lambda = 1/\alpha$ where~$\alpha$ was defined in \Cref{sec:ay-or}.

  Now we explain how to redraw \Cref{fig:bigon-track-circular}
  analogously to \Cref{fig:AY-nonor}. (Note, however, that
  \Cref{fig:bigon-track-circular} depicts the case~$g=7$, while
  \Cref{fig:AY-nonor} shows the case~$g=4$.) A large regular neighborhood 
  of~$\gamma$ on \Cref{fig:bigon-track-circular} is the disk with the crosscap in
  the middle (hence a M\"obius strip), but without the twisted bands attached.
  Since a bigon in the complement of a bigon track does not give rise to a 
  singularity of the unstable measured foliation, we can isotope the foliation so 
  that its leaves point radially inwards from the boundary of the M\"obius strip, 
  and so that the foliation is nonsingular in the interior of the M\"obius strip.  
  Note that, up to homeomorphism preserving the leaves and the transverse 
  measure of the foliation, every nonsingular measured foliation of the M\"obius 
  strip whose leaves meet the boundary transversally is the vertical foliation 
  of a rectangle with the two vertical sides identified by an involution. This is 
  the model of \Cref{fig:AY-nonor}, but without the identifications on the boundary. 
  Therefore, the only thing left to check is that attaching the twisted bands as in 
  \Cref{fig:bigon-track-circular} induces the correct identifications of the boundary 
  in the M\"obius rectangle model. 
    
  The effect of attaching the twisted bands is that intervals on the boundary
  of the M\"obius strip get identified. The lengths of these intervals with
  respect to the unstable measured foliation are the measures on the 
  branches of the bigon track inside the twisted bands, that 
  is,~$\lambda^{g-1},\lambda^{g-1}, \ldots, 1, 1$. When the M\"obius strip is drawn
  as a rectangle with its vertical sides identified, the~$2g$ intervals of
  length~$\lambda^{g-1},\lambda^{g-1}, \ldots, 1, 1$ appear on the horizontal
  sides. Note that we obtain exactly the pattern of \Cref{fig:AY-nonor}, with the 
  difference that our intervals are a factor~$\lambda^g$ longer. Since the bands 
  on \Cref{fig:bigon-track-circular} are twisted, the pairs of intervals are glued
  together in an orientation-reversing way, just like gluing by translation on
  \Cref{fig:AY-nonor} results in orientation-reversing gluings. Therefore the
  bigon track on \Cref{fig:bigon-track-circular} indeed produces the measured
  foliation pictured on \Cref{fig:AY-nonor}, up to scaling the measure by~$\lambda^g$.

  \begin{figure}[htb]
    \begin{subfigure}{0.3\textwidth}
      \labellist
      \small\hair 2pt
      \pinlabel {$\gamma$} [ ] at 122 92
      \endlabellist
      \includegraphics[width=\textwidth]{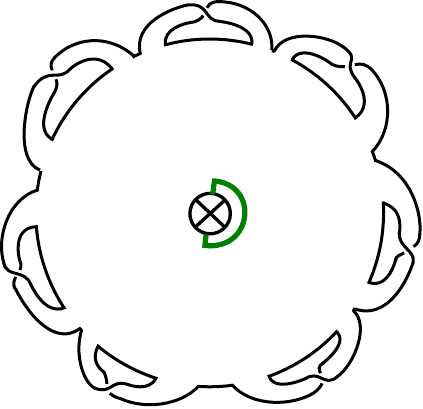}
      \caption{}
      \label{fig:core-curve}
    \end{subfigure}
    \begin{subfigure}{0.3\textwidth}
      \labellist
      \small\hair 2pt
      \pinlabel {$T_c(\gamma)$} [ ] at 97 55
      \endlabellist
      \includegraphics[width=\textwidth]{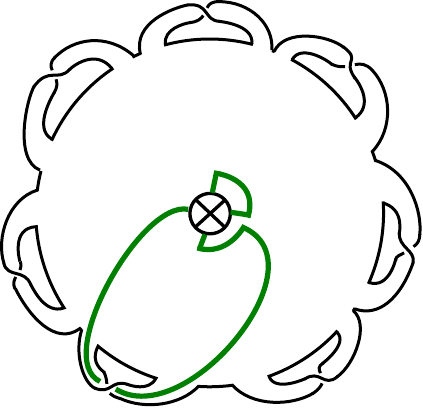}
      \caption{}
      \label{fig:twisted-core-curve}
    \end{subfigure}
    \begin{subfigure}{0.3\textwidth}
      \labellist
      \footnotesize\hair 2pt
      \pinlabel {$\gamma'=r(T_c(\gamma))$} [ ] at 90 60
      \endlabellist
      \centering
      \includegraphics[width=\textwidth]{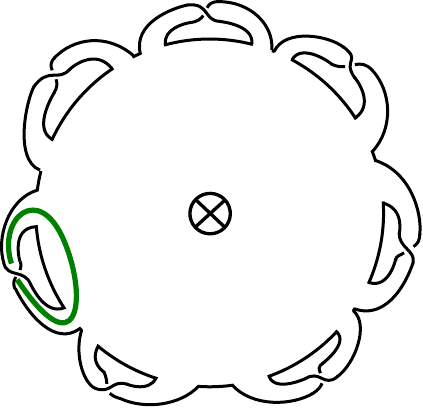}
      \caption{  }
      \label{fig:label}
    \end{subfigure}
    \caption{The curves~$\gamma$,~$T_c(\gamma)$, and~$\gamma'=r(T_c(\gamma))$.
      To go from the second figure to the third, we isotope a small piece 
      of~$T_c(\gamma)$ through the crosscap. This is possible, because the
      antipodal points of the circle containing the X are identified. The
      direction of the twisting about~$c$ is matters: one can compute directly
      that the curve~$T_c^{-1}(\gamma)$ defines an element of the fundamental group 
      that does not correspond to the core curve of one of the twisted bands.}
    \label{fig:nonor-twist}
  \end{figure}

  The curve~$\gamma' = h_g(\gamma)$ on \Cref{fig:AY-nonor} corresponds to the
  curve~$\gamma'$ on \Cref{fig:bigon-track-circular}. It
  remains to show that~$\gamma' = f_g(\gamma)$. After applying the twist~$T_c$
  on the curve~$\gamma$ on \Cref{fig:core-curve}, we obtain the curve shown on
  \Cref{fig:twisted-core-curve}. After rotation by one click, this curve indeed
  maps to~$\gamma'$.
\end{proof}

\begin{proof}[Proof of Theorem~\ref{thm:arnoux-yoccoz}]
  Consider the orientable double cover of the nonorientable surface in
  \Cref{fig:ay_penner-nonor}. One way to construct this covering surface is to
  cut along the twisted bands on \Cref{fig:bigon-track-circular}, remove the
  central crosscap, and glue together two copies of the resulting surface. We
  can think about the two copies as the upper and lower half of the the
  cylinder pictured on \Cref{fig:cylinder}. The upper and lower boundaries of
  this cylinder are subdivided into~$2g$ intervals, which correspond to
  the~$2g$ intervals obtained by cutting the twisted bands.
  \begin{figure}[htb]
    \labellist
    \small\hair 2pt
    \pinlabel {$1$} [ ] at 9 15
    \pinlabel {$2$} [ ] at 27 10
    \pinlabel {$3$} [ ] at 46 8
    \pinlabel {$4$} [ ] at 64 8
    \pinlabel {$14$} [ ] at 8 29
    \pinlabel {$13$} [ ] at 25 34
    \pinlabel {$2$} [ ] at 110 77
    \pinlabel {$1$} [ ] at 95 81
    \pinlabel {$4$} [ ] at 78 83
    \pinlabel {$3$} [ ] at 62 84
    \pinlabel {$13$} [ ] at 113 62
    \pinlabel {$14$} [ ] at 97 59
    \endlabellist
    \centering
    \includegraphics[scale=1.0]{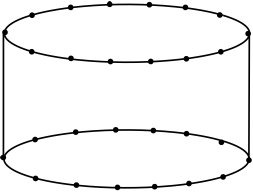}
    \caption{The orientable double cover of the surface on
      \Cref{fig:ay_penner-nonor} when~$g=7$.}
    \label{fig:cylinder}
    \end{figure}
    The orientation-reversing involution of this cylinder that identifies the
    upper and lower half is the reflection about the center of the picture in
    the ambient 3-dimensional space. When the intervals along the boundaries
    are identified in the manner shown, the quotient of the surface is our
    nonorientable surface with the twisted bands, with the boundary collapsed
    to one point.

    Note that the rotation $r$ of $N_{g+1}$ lifts to the rotation of the cylinder by two
    intervals.

    By flattening out the cylinder, we obtain the representation on
    \Cref{fig:annulus}.
\begin{figure}[htb]
  \begin{subfigure}[ht]{0.45\linewidth}
    \labellist
    \small\hair 2pt
     \pinlabel {$13$} [ ] at 10 34
     \pinlabel {$14$} [ ] at 23 16
     \pinlabel {$1$} [ ] at 44 7
     \pinlabel {$2$} [ ] at 73 7
     \pinlabel {$3$} [ ] at 97 20
     \pinlabel {$4$} [ ] at 106 36
     \pinlabel {$3$} [ ] at 45 62
     \pinlabel {$4$} [ ] at 49 66
     \pinlabel {$1$} [ ] at 55 68
     \pinlabel {$2$} [ ] at 62 68
     \pinlabel {$14$} [ ] at 73 71
     \pinlabel {$13$} [ ] at 79 64
    \endlabellist
    \centering
    \includegraphics[width=\textwidth]{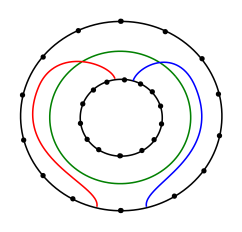}
    \caption{  }
    \label{fig:annulus}
  \end{subfigure}
  \begin{subfigure}{0.45\linewidth}
    \centering
    \includegraphics[width=\textwidth]{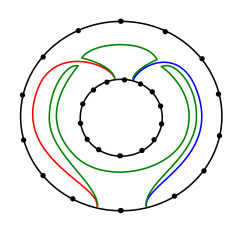}
    \caption{  }
    \label{fig:after_twist_or}
  \end{subfigure}\newline
  \begin{subfigure}[ht]{0.45\linewidth}
    \labellist
    \small\hair 2pt
    \pinlabel {$1$} at 46 6
    \pinlabel {$2$} at 75 7
    \pinlabel {$3$} at 95 19
    \pinlabel {$4$} at 105 35
    \pinlabel {$13$} at 10 32
    \pinlabel {$14$} at 23 16
    \pinlabel {$14$} at 44 50
    \pinlabel {$13$} at 47 44
    \pinlabel {$2$} at 53 41
    \pinlabel {$1$} at 62 41
    \pinlabel {$4$} at 69 43
    \pinlabel {$3$} at 73 49
    \endlabellist
    \centering
    \includegraphics[width=\textwidth]{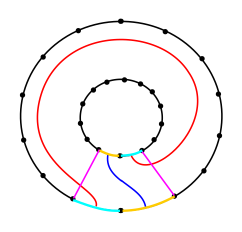}
    \caption{  }
    \label{fig:annulus-rotated}
  \end{subfigure}
  \begin{subfigure}[ht]{0.45\linewidth}
    \includegraphics[width=\textwidth]{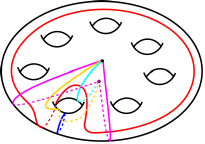}
    \caption{}
    \label{fig:S_7_chopped}
  \end{subfigure}
  \caption{}
  \end{figure}
  The lift of the curve~$c$ along which we twist in the definition of~$f_g$ has
  two components, shown on \Cref{fig:annulus}. A Dehn twist about 
  the curve $c$ on $N_{g+1}$ lifts to the product of a positive twist along one of the lifts of $c$
  and a negative twist about the other lift.

  To find out which twist is positive and which twist is negative, recall 
  that~$T_c(\gamma)$ is a curve that runs in a small neighborhood of one of the
  twisted bands (\Cref{fig:core-curve,fig:twisted-core-curve}). The core curve
  of the annulus on \Cref{fig:annulus} is the lift of~$\gamma$, so its image
  under the two twists should run in a small neighborhood of two consecutive
  intervals of the boundary of the annulus. That happens when the twist is
  positive along the curve on the right and negative along the curve on the
  left on \Cref{fig:after_twist_or}.

  After changing \Cref{fig:annulus} by rotating the inner boundary by 180
  degrees, we obtain the representation shown on \Cref{fig:annulus-rotated}. By
  subdividing this surface along the arcs shown and their rotated copies, we
  can see that this surface can be represented in~$\R^3$ as the surface on
  \Cref{fig:S_7_chopped}. The two twisting curves correspond to the curves
  shown on \Cref{fig:ay_penner}, and we indeed twist positively along the 
  curve~$a$ and negatively along the curve~$b$.
\end{proof}

\subsection*{Remarks on the order of composition and the direction of twisting}

Recall from \Cref{thm:arnoux-yoccoz} the definition of the curves~$a$ and~$b$
and the rotation~$r$. Consider also the curve~$b'$ on \Cref{fig:other-wind}
that winds around the hole to avoid~$a$ in the direction opposite of~$b$.

\begin{figure}[ht]
  \labellist
  \small\hair 2pt
  \pinlabel {$a$} at 33 8
  \pinlabel {$b'$} at 48 12
  \endlabellist
  \centering
  \includegraphics[width=0.4\textwidth]{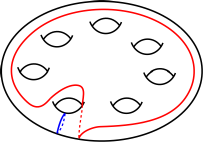}
  \caption{}
  \label{fig:other-wind}
\end{figure}

The following statement summarizes the ways in which the definition
of the mapping class~$\tilde{f}_g = r \circ T_a \circ T_b^{-1}$ is flexible.

\begin{proposition}
  The following statements hold.
  \begin{enumerate}
  \item $T_a$ commutes with both $T_b^{-1}$ and $T_{b'}^{-1}$.
  \item $\tilde{f}_g$ is conjugate to $r^{\pm 1} \circ T_a \circ T_b^{-1}$,
    $T_a \circ T_b^{-1} \circ r^{\pm1}$,
    $r^{\pm 1} \circ T_a^{-1} \circ T_{b'}$ and $T_a^{-1} \circ T_{b'} \circ r^{\pm 1}$.
  \item $\tilde{f}_g^{-1}$ is conjugate to
    $T_a^{-1} \circ T_b \circ r^{\pm1}$, $r^{\pm 1} \circ T_a^{-1} \circ T_b$,
    $T_a \circ T_{b'}^{-1} \circ r^{\pm 1}$ and $r^{\pm 1} \circ T_a \circ T_{b'}^{-1}$.
  \end{enumerate}
\end{proposition}
\begin{proof}
  The first statement holds, because~$a$ is disjoint from~$b$ and $b'$.

  For the first expression in the second statement, note that on
  \Cref{fig:ay_penner}, the rotation by~$180$ degrees about the axis
  intersecting~$a$ and~$b$ symmetrically commutes with both~$T_a$
  and~$T_b^{-1}$, but conjugates~$r$ to~$r^{-1}$. The second expression is
  conjugate to the first by~$r$. For the third and fourth expressions, note
  that the reflection about the plane that intersects all~$g$ holes of the
  surface commutes with~$r$ and conjugates~$T_a$ to~$T_a^{-1}$ and~$T_b$ 
  to~$T_{b'}^{-1}$.

  The third statement follows from the second by taking the inverse.
\end{proof}

To summarize, the order of the two twists, the direction of the rotation and
whether we twist first or rotate first do not matter.

In the nonorientable case, we have the following.

\begin{proposition}
  The mapping class~$f_g$ is conjugate to~$r^{\pm 1} \circ T_c$ and~$T_c \circ r^{\pm 1}$. 
  The inverse~$f_g^{-1}$ is conjugate to~$T_c^{-1} \circ r^{\pm 1}$
  and~$r^{\pm 1} \circ T_c^{-1}$.
\end{proposition}
\begin{proof}
  This follows by conjugating by~$r$ and by an involution of the surface that
  rotates about an axis and leaves the curve~$c$ invariant. This involution
  commutes with~$T_c$ and conjugates~$r$ to~$r^{-1}$.
\end{proof}

We also remark that it can be shown by studying their flat surfaces that~$h_g$
is conjugate to~$h_g^{-1}$ and~$\tilde{h}_g$ is conjugate to~$\tilde{h}_g^{-1}$
when~$g=3$, but not if~$g>3$. Therefore one indeed needs to be careful about
the definitions, because the direction in which~$b$ winds around the hole does
matter.

\bibliographystyle{alpha}
\bibliography{../mybibfile}

\end{document}